\documentclass[11pt]{amsart}
\usepackage{amsmath}
\usepackage{amssymb}
\usepackage{amsfonts}

\newtheorem{theorem}{Theorem}[section]
\newtheorem{corollary}[theorem]{Corollary}

\newtheorem{definition}[theorem]{Definition}
\newtheorem{proposition}[theorem]{Proposition}
\newtheorem{remark}[theorem]{Remark}
\newtheorem{example}[theorem]{Example}
\numberwithin{theorem}{section}

\begin{document}
\title[The Shilov boundary for a local operator system]{The Shilov boundary
for a local operator system}
\author{Maria Joi\c{t}a}
\address{Department of Mathematics, Faculty of Applied Sciences, University
Politehnica of Bucharest, 313 Spl. Independentei, 060042, Bucharest, Romania
and Department of Mathematics, Faculty of Mathematics and Computer Science,
University of Bucharest, Str. Academiei nr. 14, Bucharest, Romania}
\email{maria.joita@upb.ro and mjoita@fmi.unibuc.ro }
\urladdr{http://sites.google.com/a/g.unibuc.ro/maria-joita}
\thanks{This work was supported by a grant of UEFISCDI, project number
PN-III-P4-PCE-2021-0282}
\thanks{This paper is in final form and no version of it will be submitted
for publication elsewhere.}

\begin{abstract}
In this paper, we introduce the notion of Shilov boundary ideal for a local
operator system and investigate some of its properties.
\end{abstract}

\maketitle

\section{Introduction}

Let $\mathcal{M}$ be a uniformly linear subspace of $C(X)$, the $C^{\ast }$%
-algebra of all continuous complex valued functions on a compact Hausdorff
space $X$, which separates the points of $X$ and contains constants. A
subset $K$ of $X$ with the property that any function form $\mathcal{M}$
achieves its maximum modulus on $K$ is called a boundary for $\mathcal{M}$.
If\textbf{\ }$K$ is a boundary for $\mathcal{M}$, then $\mathcal{J}=\{f\in
C(X)/\left. f\right\vert _{K}=0\}$ is a closed two-sided $\ast $-ideal of $%
C(X)$ and the quotient map $\sigma :$ $C(X)\rightarrow C(X)/\mathcal{J}$ is
completely isometric on $\mathcal{M}$. The Shilov boundary for $\mathcal{M}$
is the smallest closed boundary. Arveson \cite{A} extends these notions in
the non-commutative case. A concrete operator system is a self adjoint
subspace $\mathcal{S}$ of a unital $C^{\ast }$-algebra $\mathcal{A}$ which
contains the unit of $\mathcal{A}$. Suppose that the $C^{\ast }$-subalgebra $%
C^{\ast }(\mathcal{S})$ of $\mathcal{A}$ generated by $\mathcal{S}$
coincides to $\mathcal{A}$. A boundary representation for $\mathcal{S}$ is
an irreducible representation $\pi $ of $\mathcal{A}$ with the property that 
$\pi $ is the unique unital completely positive map which extends $\left.
\pi \right\vert _{\mathcal{S}}$. Boundary representations are the analogues
of Choquet points in the commutative case. A closed two-sided $\ast $-ideal $%
\mathcal{J}$ of $\mathcal{A}$ is a boundary ideal for $\mathcal{S}$ if the
quotient map $\sigma :\mathcal{A}\rightarrow $ $\mathcal{A}/\mathcal{J}$ is
completely isometric on $\mathcal{S}$. A boundary ideal that contains all
other boundary ideals is called the Shilov ideal for $\mathcal{S}$. The
existence of boundary representations was an open problem for almost 40
years. In 2007, Arveson \cite{A2} shows that a separable operator system $%
\mathcal{S}$ has sufficiently many boundary representations, and the Shilov
ideal for $\mathcal{S}$ is intersection of the kernels of boundary
representations. G. Kleski \cite{K} improves the Arveson's result.

Effros and Webster \cite{EW} initiated a study of the locally convex version
of operator spaces called the local operator spaces. A locally $C^{\ast }$%
-algebra is a complete Hausdorff topological $\ast $-algebra whose topology
is defined by a directed family of $C^{\ast }$-seminorms. The $\ast $%
-algebras of unbounded linear operators are concrete models for locally $%
C^{\ast }$-algebras. A quantized domain\textit{\ }in a Hilbert space $%
\mathcal{H}$ is a triple $\{\mathcal{H};\mathcal{E};\mathcal{D}_{\mathcal{E}%
}\}$, where $\mathcal{E}=\{\mathcal{H}_{\delta };\delta \in \Delta \}$ is an
upward filtered family of closed subspaces of $\mathcal{H}$ with dense union 
$\mathcal{D}_{\mathcal{E}}=\bigcup\limits_{\delta \in \Delta }\mathcal{H}%
_{\delta }$ in $\mathcal{H}$. The collection of all linear operators $T:%
\mathcal{D}_{\mathcal{E}}\rightarrow \mathcal{D}_{\mathcal{E}}$ such that $T(%
\mathcal{H}_{\delta })\subseteq \mathcal{H}_{\mathcal{\delta }},T(\mathcal{H}%
_{\delta }^{\bot }\cap \mathcal{D}_{\mathcal{E}})\subseteq \mathcal{H}%
_{\delta }^{\bot }\cap \mathcal{D}_{\mathcal{E}}$ $\ $and$\left.
T\right\vert _{\mathcal{H}_{\delta }}\in B(\mathcal{H}_{\delta })$ for all $%
\delta \in \Delta ,$ denoted by $C^{\ast }(\mathcal{D}_{\mathcal{E}}),$ is a
locally $C^{\ast }$-algebra with the involution given by $T^{\ast }=\left.
T^{\bigstar }\right\vert _{\mathcal{D}_{\mathcal{E}}}$, where $T^{\bigstar }$
is the adjoint of the unbounded linear operator $T,$\ and the topology given
by the family of $C^{\ast }$-seminorms $\{\left\Vert \cdot \right\Vert
_{\delta }\}_{\delta \in \Delta }$, where $\left\Vert T\right\Vert _{\delta
}=\left\Vert \left. T\right\vert _{\mathcal{H}_{\delta }}\right\Vert _{B(%
\mathcal{H}_{\delta })}$. For every locally $C^{\ast }$-algebra $\mathcal{A}$
whose topology is defined by the family of $C^{\ast }$-seminorms $%
\{p_{\lambda }\}_{\lambda \in \Lambda }$, there is a quantized domain $\{%
\mathcal{H};\mathcal{E=}\{\mathcal{H}_{\lambda }\}_{\lambda \in \Lambda };%
\mathcal{D}_{\mathcal{E}}\}$ and a $\ast $-morphism $\pi :A\mathcal{%
\rightarrow }C^{\ast }(\mathcal{D}_{\mathcal{E}})$ such that $\left\Vert \pi
\left( a\right) \right\Vert _{\lambda }=p_{\lambda }\left( a\right) \ $for
all $a\in A$ and for all $\lambda \in \Lambda $ (Dosiev \cite{D} called it a
local isometric $\ast $-morphism). Dosiev \cite{DA} shows that a local
operator space can be realized as a subspace of $C^{\ast }(\mathcal{D}_{%
\mathcal{E}})$. Arunkumar \cite{Ar} considers the representations of locally 
$C^{\ast }$-algebras on quantized domains and he introduces the notion of
local boundary representation for a local operator system. In \cite{J24}, we
shows that the local boundary representations considered by Arunkumar are in
fact representations on Hilbert spaces. In this paper, we introduce the
notion of local Shilov boundary ideal for a local operator system and we
characterize it in terms of its admissible local boundary representations
(Theorem \ref{A}). Also, we show that, in the case of separable Fr\'{e}chet
local operators systems with trivial local Shilov boundary ideals, any
surjective unital local completely isometry between two local operator
systems is implemented by a local contractive $\ast $-isomorphism between
the locally $C^{\ast }$-algebras generated by them (Theorem \ref{D}).

\section{Preliminaries}

A \textit{\ $C^{\ast }$-convex algebra} $\mathcal{A}$ is a Hausdorff
topological $\ast $-algebra over $\mathbb{C}$ whose topology is determined
by an upward filtered family $\left\{ p_{\lambda }\right\} _{\lambda \in
\Lambda }$ of $C^{\ast }$-seminorms defined on $\mathcal{A}$ (that means, if 
$\lambda _{1}\leq \lambda _{2},$ then $p_{\lambda _{1}}\left( a\right) \leq $
$p_{\lambda _{2}}\left( a\right) $ for all $a\in \mathcal{A}$). A complete $%
C^{\ast }$-convex algebra is called a \textit{locally $C^{\ast }$-algebra }%
\cite{I}. A \textit{Fr\'{e}chet locally }$C^{\ast }$\textit{-algebra }is a
locally $C^{\ast }$-algebra whose topology is determined by a countable
family of $C^{\ast }$-seminorms. The topology on a Fr\'{e}chet locally $%
C^{\ast }$-algebra is unique.

A locally $C^{\ast }$-algebra $\mathcal{A}$ can be realized as a projective
limit of an inverse family of $C^{\ast }$-algebras. If $\mathcal{A}$ is a
locally $C^{\ast }$-algebra\textit{\ }with the topology determined by the
family of $C^{\ast }$-seminorms $\{p_{\lambda }\}_{\lambda \in \Lambda }$,
for each $\lambda \in \Lambda $, $\ker p_{\lambda }=\{a\in \mathcal{A}%
;p_{\lambda }\left( a\right) =0\}$ is a closed two sided $\ast $-ideal in $%
\mathcal{A}$ and $\mathcal{A}_{\lambda }=\mathcal{A}/\ker p_{\lambda }$ is a 
$C^{\ast }$-algebra with respect to the $C^{\ast }$-norm induced by $%
p_{\lambda }$. The canonical quotient $\ast $-morphism from $\mathcal{A}$ to 
$\mathcal{A}_{\lambda }$ is denoted by $\pi _{\lambda }^{\mathcal{A}}$. For
each $\lambda _{1},\lambda _{2}\in \Lambda $ with $\lambda _{1}\leq \lambda
_{2}$, there is a canonical surjective $\ast $-morphism $\pi _{\lambda
_{2}\lambda _{1}}^{\mathcal{A}}:$ $\mathcal{A}_{\lambda _{2}}\rightarrow 
\mathcal{A}_{\lambda _{1}}$, defined by $\pi _{\lambda _{2}\lambda _{1}}^{%
\mathcal{A}}\left( a+\ker p_{\lambda _{2}}\right) =a+\ker p_{\lambda _{1}}$
for $a\in \mathcal{A}$. Then, $\{\mathcal{A}_{\lambda },\pi _{\lambda
_{2}\lambda _{1}}^{\mathcal{A}}\}_{\lambda _{1}\leq \lambda _{2},\lambda
,\lambda _{1},\lambda _{2}\in \Lambda }$\ forms an inverse system of $%
C^{\ast }$-algebras, since $\pi _{\lambda _{1}}^{\mathcal{A}}=$ $\pi
_{\lambda _{2}\lambda _{1}}^{\mathcal{A}}\circ \pi _{\lambda _{2}}^{\mathcal{%
A}}$ whenever $\lambda _{1}\leq \lambda _{2}$. The projective limit%
\begin{equation*}
\lim\limits_{\underset{\lambda }{\leftarrow }}\mathcal{A}_{\lambda
}:=\{\left( a_{\lambda }\right) _{\lambda \in \Lambda }\in
\prod\limits_{\lambda \in \Lambda }\mathcal{A}_{\lambda };\pi _{\lambda
_{2}\lambda _{1}}^{\mathcal{A}}\left( a_{\lambda _{2}}\right) =a_{\lambda
_{1}}\text{ whenever }\lambda _{1}\leq \lambda _{2},\lambda _{1},\lambda
_{2}\in \Lambda \}
\end{equation*}%
of the inverse system of $C^{\ast }$-algebras $\{\mathcal{A}_{\lambda },\pi
_{\lambda _{2}\lambda _{1}}^{\mathcal{A}}\}_{\lambda _{1}\leq \lambda
_{2},\lambda ,\lambda _{1},\lambda _{2}\in \Lambda }$ is a locally $C^{\ast
} $-algebra that can be identified with $\mathcal{A}$ by the map $a\mapsto
\left( \pi _{\lambda }^{\mathcal{A}}\left( a\right) \right) _{\lambda \in
\Lambda }$.

Let $\mathcal{A}$ be a locally $C^{\ast }$-algebra whose topology is defined
by the family of $C^{\ast }$-seminorms $\{p_{\lambda }\}_{\lambda \in
\Lambda }.$

An element $a\in \mathcal{A}$ is called \textit{local self-adjoint} if $%
a=a^{\ast }+c$, where $c\in \mathcal{A}$ such that $p_{\lambda }\left(
c\right) =0$ for some $\lambda \in \Lambda $, and we call $a$ as $\lambda $%
\textit{-self-adjoint,} and \textit{local positive} if $a=b^{\ast }b+c$,
where $b,c\in $ $\mathcal{A}$ such that $p_{\lambda }\left( c\right) =0\ $
for some $\lambda \in \Lambda $, we call $a$ as $\lambda $\textit{-positive }%
and write\textit{\ }$a\geq _{\lambda }0$\textit{. }We\textit{\ }write $%
a=_{\lambda }0$ whenever $p_{\lambda }\left( a\right) =0$. Note that $a\in 
\mathcal{A}$ is local self-adjoint if and only if there is $\lambda \in
\Lambda $ such that $\pi _{\lambda }^{\mathcal{A}}\left( a\right) $ is self
adjoint in $\mathcal{A}_{\lambda }$ and $a\in \mathcal{A}$ is local positive
if and only if there is $\lambda \in \Lambda $ such that $\pi _{\lambda }^{%
\mathcal{A}}\left( a\right) $ is positive in $\mathcal{A}_{\lambda }.$

A \textit{local representation} of $\mathcal{A}$ on a Hilbert space $%
\mathcal{H}$ is a $\ast $-morphism $\pi :\mathcal{A\rightarrow }B(\mathcal{H}%
)$ with the property that there exists $\lambda \in \Lambda $ such that $%
\left\Vert \pi \left( a\right) \right\Vert \leq p_{\lambda }\left( a\right) $
for all $a\in \mathcal{A}$. So, $\pi :\mathcal{A\rightarrow }B(\mathcal{H})$
is a local representation for $\mathcal{A}$ if and only if $\pi $ is a
continuous $\ast $-morphism from $\mathcal{A}$ to $B(\mathcal{H})$ and it is
called a continuous $\ast $-representation of $\mathcal{A}$ on $\mathcal{H}$
in \cite[\S\ 13]{Fr}. A local representation $\pi :\mathcal{A\rightarrow }B(%
\mathcal{H})$ is irreducible if $\pi \left( A\right) ^{\prime }=\mathbb{C}$id%
$_{\mathcal{H}},$ where $\pi \left( A\right) ^{\prime }$ is the commutant of 
$\pi \left( A\right) $ in $B(\mathcal{H}).$ Moreover, a local representation 
$\pi :\mathcal{A\rightarrow }B(\mathcal{H})$ is irreducible if and only if
there exist $\lambda _{0}\in \Lambda $ and an irreducible representation $%
\pi _{\lambda _{0}}:\mathcal{A}_{\lambda _{0}}\mathcal{\rightarrow }B(%
\mathcal{H})$ of $\mathcal{A}_{\lambda _{0}}$ such that $\pi =\pi _{\lambda
_{0}}\circ \pi _{\lambda _{0}}^{\mathcal{A}}.$

A \textit{local operator system} is a self-adjoint subspace $\mathcal{S}$ of
a unital locally $C^{\ast }$-algebra $\mathcal{A}$ which contains the unity
of $\mathcal{A}$. For each $\lambda \in \Lambda ,$ $\mathcal{S}_{\lambda
}:=\pi _{\lambda }^{\mathcal{A}}\left( \mathcal{S}\right) $ is an operator
system in $\mathcal{A}_{\lambda },$ and $\mathcal{S}=\lim\limits_{\underset{%
\lambda }{\leftarrow }}\mathcal{S}_{\lambda }$. Moreover, if $\mathcal{S}$
generates $\mathcal{A}$, then, for each $\lambda \in \Lambda ,$ $\mathcal{S}%
_{\lambda }$ generates $\mathcal{A}_{\lambda }$.

For each $n\geq 1,$ $M_{n}(\mathcal{A})$ denotes the collection of all
matrices of order $n$ with elements in $\mathcal{A}$. Note that $M_{n}(%
\mathcal{A})$ is a locally $C^{\ast }$-algebra, the associated family of $%
C^{\ast }$-seminorms being denoted by $\{p_{\lambda }^{n}\}_{\lambda \in
\Lambda }$, and $p_{\lambda }^{n}\left( \left[ a_{ij}\right]
_{i,j=1}^{n}\right) =\left\Vert \left[ \pi _{\lambda }^{\mathcal{A}}\left(
a_{ij}\right) \right] _{i,j=1}^{n}\right\Vert _{M_{n}(\mathcal{A}_{\lambda
})}.$

For each $n\in \mathbb{N}$, the $n$-amplification of a linear map $\varphi :%
\mathcal{S}\rightarrow B(\mathcal{H})$ is the map $\varphi ^{\left( n\right)
}:M_{n}(\mathcal{S})$ $\rightarrow $ $B(\mathcal{H})$ defined by 
\begin{equation*}
\varphi ^{\left( n\right) }\left( \left[ s_{ij}\right] _{i,j=1}^{n}\right) =%
\left[ \varphi \left( s_{ij}\right) \right] _{i,j=1}^{n}
\end{equation*}%
for all $\left[ s_{ij}\right] _{i,j=1}^{n}\in M_{n}(\mathcal{S})$ .

An element $\left[ s_{ij}\right] _{i,j=1}^{n}$ is positive, respectively $%
\lambda $-positive in $M_{n}(\mathcal{S}),n\geq 1$ if it is positive,
respectively $\lambda $ -positive in $M_{n}(\mathcal{A}),n\geq 1.$

\begin{definition}
A linear map $\varphi :\mathcal{S}\rightarrow B(\mathcal{H})$ is called:

\begin{enumerate}
\item \textit{local positive} if there is $\lambda \in \Lambda $ such that $%
\varphi \left( a\right) \ $is positive in $B\left( \mathcal{H}\right) \ $
whenever $a$ is $\lambda $-positive\textit{\ } and $\varphi \left( a\right)
=0\ $whenever $a=_{\lambda }0;$

\item \textit{local completely positive (local }$\mathcal{CP}$\textit{) }if\
there\ is\ $\lambda \in \Lambda \ $ such\ that $\ \varphi ^{\left( n\right)
}\left( \left[ a_{ij}\right] _{i,j=1}^{n}\right) \ $is positive in $B\left( 
\mathcal{H}^{\oplus n}\right) $ whenever $\left[ \left( a_{ij}\right) \right]
_{i,j=1}^{n}$ is $\lambda $-positive\textit{\ } and $\varphi ^{\left(
n\right) }\left( \left[ a_{ij}\right] _{i,j=1}^{n}\right) =0\ \ $whenever $%
\left[ a_{ij}\right] _{i,j=1}^{n}=_{\lambda }0$,$\ $for all $n\geq 1$.
\end{enumerate}
\end{definition}

\section{Main results}

\subsection{Local completely isometric maps}

Let $\mathcal{A}_{1}$ and $\mathcal{A}_{2}$ be two unital locally $C^{\ast }$%
-algebras with the topologies given by the families of $C^{\ast }$-seminorms 
$\{p_{\lambda }\}_{\lambda \in \Lambda },$ respectively $\{q_{\lambda
}\}_{\lambda \in \Lambda },$ $\mathcal{S}_{1}\subseteq \mathcal{A}_{1}$ and $%
\mathcal{S}_{2}\subseteq \mathcal{A}_{2}$ be two local operator systems.

\begin{definition}
A linear map $\varphi :\mathcal{S}_{1}\rightarrow \mathcal{S}_{2}$ is called
an admissible \textit{local completely positive map if for each }$\lambda
\in \Lambda ,$ $\varphi ^{\left( n\right) }\left( \left[ s_{ij}\right]
_{i,j=1}^{n}\right) $ is $\lambda $-positive whenever $\left[ s_{ij}\right]
_{i,j=1}^{n}$ is $\lambda $-positive and $\varphi ^{\left( n\right) }\left( %
\left[ s_{ij}\right] _{i,j=1}^{n}\right) =_{\lambda }0$ whenever $\left[
s_{ij}\right] _{i,j=1}^{n}$ $=_{\lambda }0$ for all $n.$
\end{definition}

\begin{proposition}
\label{x}A linear map $\varphi :\mathcal{S}_{1}\rightarrow \mathcal{S}_{2}$
is admissible \textit{local completely positive }if and only if there exists
an inverse system $(\varphi _{\lambda })_{\lambda }$ of completely positive
maps such that $\varphi =\lim\limits_{\underset{\lambda }{\leftarrow }%
}\varphi _{\lambda }.$
\end{proposition}

\begin{proof}
If $\varphi :\mathcal{S}_{1}\rightarrow \mathcal{S}_{2}$ is a local
completely positive map, then for each $\lambda \in \Lambda $, $\pi
_{\lambda }^{\mathcal{A}_{2}}\circ \varphi :\mathcal{S}_{1}\rightarrow
\left( \mathcal{S}_{2}\right) _{\lambda }$ is a local completely positive
map such that $\left\Vert \left( \pi _{\lambda }^{\mathcal{A}_{2}}\circ
\varphi \right) \left( s_{1}\right) \right\Vert _{\left( \mathcal{A}%
_{2}\right) _{\lambda }}=p_{\lambda }\left( s_{1}\right) $ for all $s_{1}\in 
\mathcal{S}_{1}$, and by \cite[Renark 3.1]{JM22}, there exists a completely
positive map $\varphi _{\lambda }:\left( \mathcal{S}_{1}\right) _{\lambda
}\rightarrow \left( \mathcal{S}_{2}\right) _{\lambda }$ such that $\varphi
_{\lambda }\circ \left. \pi _{\lambda }^{\mathcal{A}_{1}}\right\vert _{%
\mathcal{S}_{1}}=\pi _{\lambda }^{\mathcal{A}_{2}}\circ \varphi $. Clearly, $%
(\varphi _{\lambda })_{\lambda }$ is an inverse system of completely
positive maps and $\varphi =\lim\limits_{\underset{\lambda }{\leftarrow }%
}\varphi _{\lambda }$.

The converse implication is trivial.
\end{proof}

\begin{definition}
A linear map $\varphi :\mathcal{S}_{1}\rightarrow \mathcal{S}_{2}$ is called
a \textit{local} \textit{completely isometric} map if 
\begin{equation*}
q_{\lambda }^{n}\left( \left[ \varphi \left( s_{ij}\right) \right]
_{i,j=1}^{n}\right) =p_{\lambda }^{n}\left( \left[ s_{ij}\right]
_{i,j=1}^{n}\right)
\end{equation*}%
for all$\ \left[ s_{ij}\right] _{i,j=1}^{n}\in M_{n}\left( \mathcal{S}%
_{1}\right) $, for all $n\geq 1$ and for all $\lambda \in \Lambda .$
\end{definition}

\begin{remark}
If $\varphi :\mathcal{S}_{1}\rightarrow \mathcal{S}_{2}$ is local completely
isometric, then there exists a linear map\textit{\ }$\varphi ^{-1}$ $%
:\varphi \left( \mathcal{S}_{1}\right) \rightarrow \mathcal{S}_{1}$ such
that $\varphi ^{-1}\circ \varphi =$id$_{\mathcal{S}_{1}}$. Moreover, $%
\varphi ^{-1}$ is a local completely isometrc map.
\end{remark}

\begin{proposition}
A linear map $\varphi :\mathcal{S}_{1}\rightarrow \mathcal{S}_{2}$ is local
completely isometric if and only if there exists an inverse system $(\varphi
_{\lambda })_{\lambda }$ of completely isometric maps such that $\varphi
=\lim\limits_{\underset{\lambda }{\leftarrow }}\varphi _{\lambda }.$
\end{proposition}

\begin{proof}
It is similar to the proof of Proposition \ref{x}.
\end{proof}

From the above proposition and taking into account that a unital linear map $%
\varphi $ from an operator system $\mathcal{S}_{1}$ to an operator system $%
\mathcal{S}_{2}$ is a completely isometric map if and only if $\varphi $ is
an isometry and both $\varphi $ and $\varphi ^{-1}$ are completely positive
maps, we obtain the following result.

\begin{corollary}
Let $\varphi :\mathcal{S}_{1}\rightarrow \mathcal{S}_{2}$ be a unital linear
map. Then, $\varphi $ is a local completely isometric map if and only if $%
\varphi $ is a local isometry and both $\varphi $ and $\varphi ^{-1}:\varphi
\left( \mathcal{S}_{1}\right) \rightarrow \mathcal{S}_{1}$ are admissible
local completely positive maps.
\end{corollary}

\begin{remark}
If $\varphi :\mathcal{A}_{1}\rightarrow \mathcal{A}_{2}$ is a local
isometric $\ast $-morphism, then $\varphi $ is local completely isometric.
\end{remark}

\subsection{Admissible local boundary representations}

Let $\mathcal{A}$ be a unital locally $C^{\ast }$-algebra whose topology is
given the family of $C^{\ast }$-seminorms $\{p_{\lambda }\}_{\lambda \in
\Lambda }$ and $\mathcal{S}\subseteq \mathcal{A}$ be a local operator system
such that the locally $C^{\ast }$-subalgebra $C^{\ast }(\mathcal{S})$ of $%
\mathcal{A}$ generated by $\mathcal{S}$ coincides to $\mathcal{A}$.

\begin{remark}
Let $\varphi :\mathcal{S}\rightarrow B(\mathcal{H})$ be a unital local $%
\mathcal{CP}$-map. By \cite[Remark 3.1]{JM22} there exist $\lambda _{0}\in
\Lambda $ and a unital $\mathcal{CP}$-map $\varphi _{\lambda _{0}}:\mathcal{S%
}_{\lambda _{0}}\rightarrow B\mathcal{(H)}$ such that $\varphi =\varphi
_{\lambda _{0}}\circ \left. \pi _{\lambda _{0}}^{\mathcal{A}}\right\vert _{%
\mathcal{S}}\ $. Since $\varphi _{\lambda _{0}}:\mathcal{S}_{\lambda
_{0}}\rightarrow B\mathcal{(H)}$ is completely positive, by Arveson's
extension theorem \cite[Theorem 1.2.3]{A}, there exists a unital $\mathcal{CP%
}$-map $\widetilde{\varphi _{\lambda _{0}}}:\mathcal{A}_{\lambda
_{0}}\rightarrow B\mathcal{(H)}$ such that $\left. \widetilde{\varphi
_{\lambda _{0}}}\right\vert _{\mathcal{S}_{\lambda _{0}}}=\varphi _{\lambda
_{0}}$. Let $\widetilde{\varphi }=$ $\widetilde{\varphi _{\lambda _{0}}}%
\circ \pi _{\lambda _{0}}^{\mathcal{A}}$. Clearly, $\widetilde{\varphi }$ is
a local unital $\mathcal{CP}$-map from $\mathcal{A}$ to $B(\mathcal{H})$
such that $\left. \widetilde{\varphi }\right\vert _{\mathcal{S}}=\varphi $
and 
\begin{equation*}
\left\Vert \widetilde{\varphi }\left( a\right) \right\Vert =\left\Vert 
\widetilde{\varphi _{\lambda _{0}}}\left( \pi _{\lambda _{0}}^{\mathcal{A}%
}\left( a\right) \right) \right\Vert \leq \left\Vert \pi _{\lambda _{0}}^{%
\mathcal{A}}\left( a\right) \right\Vert _{\mathcal{A}_{\lambda _{0}}}=
p_{\lambda _{0}}\left( a\right)
\end{equation*}%
for all $a\in \mathcal{A}$.
\end{remark}

\begin{definition}
A linear map $\pi :\mathcal{A}\rightarrow B(\mathcal{H})$ is an admissible
local boundary representation for $\mathcal{S}$ if:

\begin{enumerate}
\item $\pi $ is a local irreducible representation;

\item there exists $\lambda _{0}\in \Lambda $ such that $\pi $ is the unique
unital local completely positive map that extends $\left. \pi \right\vert _{%
\mathcal{S}}$ and $\left\Vert \pi \left( a\right) \right\Vert \leq
p_{\lambda _{0}}\left( a\right) $ for all $a\in \mathcal{A}$.
\end{enumerate}
\end{definition}

\begin{proposition}
\label{c} Let $\pi :\mathcal{A}\rightarrow B(\mathcal{H})$ be a linear map.
Then, $\pi $ is an admissible local boundary representation for $\mathcal{S}$
if and only if there exist $\lambda _{0}\in \Lambda $ and a boundary
representation $\pi _{\lambda _{0}}:\mathcal{A}_{\lambda _{0}}\rightarrow B(%
\mathcal{H})$ for $\mathcal{S}_{\lambda _{0}}$ such that $\pi =\pi _{\lambda
_{0}}\circ \pi _{\lambda _{0}}^{\mathcal{A}}$.
\end{proposition}

\begin{proof}
Let $\pi :\mathcal{A}\rightarrow B(\mathcal{H})$ be an admissible local
boundary representation for $\mathcal{S}$. Then there exist $\lambda _{0}\in
\Lambda $ and an irreducible representation $\pi _{\lambda _{0}}:\mathcal{A}%
_{\lambda _{0}}\rightarrow B(\mathcal{H})$ such that $\pi =\pi _{\lambda
_{0}}\circ \pi _{\lambda _{0}}^{\mathcal{A}}$. Let $\varphi _{\lambda _{0}}:%
\mathcal{A}_{\lambda _{0}}\rightarrow B(\mathcal{H})$ be a unital $\mathcal{%
CP}$-map such that $\left. \varphi _{\lambda _{0}}\right\vert _{\mathcal{S}%
_{\lambda _{0}}}=\left. \pi _{\lambda _{0}}\right\vert _{\mathcal{S}%
_{\lambda _{0}}}$. Then $\varphi =\varphi _{\lambda _{0}}\circ \pi _{\lambda
_{0}}^{\mathcal{A}}$ is a unital local $\mathcal{CP}$-map. Moreover, 
\begin{equation*}
\left\Vert \varphi \left( a\right) \right\Vert =\left\Vert \left( \varphi
_{\lambda _{0}}\circ \pi _{\lambda _{0}}^{\mathcal{A}}\right) \left(
a\right) \right\Vert \leq \left\Vert \pi _{\lambda _{0}}^{\mathcal{A}}\left(
a\right) \right\Vert _{\mathcal{A}_{\lambda _{0}}}=p_{\lambda _{0}}\left(
a\right)
\end{equation*}%
for all $a\in \mathcal{A}$ and%
\begin{equation*}
\left. \varphi \right\vert _{\mathcal{S}}=\left. \varphi _{\lambda
_{0}}\right\vert _{\mathcal{S}_{\lambda _{0}}}\circ \left. \pi _{\lambda
_{0}}^{\mathcal{A}}\right\vert _{\mathcal{S}}=\left. \pi _{\lambda
_{0}}\right\vert _{\mathcal{S}_{\lambda _{0}}}\circ \left. \pi _{\lambda
_{0}}^{\mathcal{A}}\right\vert _{\mathcal{S}}=\left. \pi \right\vert _{%
\mathcal{S}}
\end{equation*}%
whence, since $\pi $ is an admissible local boundary representation for $%
\mathcal{S}$, we deduce that $\varphi =\pi $. Consequently, $\varphi
_{\lambda _{0}}=\pi _{\lambda _{0}},$ and so, $\pi _{\lambda _{0}}$ is a
boundary representation for $\mathcal{S}_{\lambda _{0}}$.

Conversely, suppose that there exist $\lambda _{0}\in \Lambda $ and a
boundary representation $\pi _{\lambda _{0}}:\mathcal{A}_{\lambda
_{0}}\rightarrow B(\mathcal{H})$ for $\mathcal{S}_{\lambda _{0}}$ such that $%
\pi =\pi _{\lambda _{0}}\circ \pi _{\lambda _{0}}^{\mathcal{A}}$. Then $\pi $
is an irreducible local representation of $\mathcal{A}$. Let $\varphi :%
\mathcal{A}\rightarrow B(\mathcal{H})$ be a unital local $\mathcal{CP}$-map
such that $\left\Vert \varphi \left( a\right) \right\Vert \leq p_{\lambda
_{0}}\left( a\right) $ for all $a\in \mathcal{A}$ and $\left. \varphi
\right\vert _{\mathcal{S}}=\left. \pi \right\vert _{\mathcal{S}}$. Then
there exists a unital $\mathcal{CP}$-map $\varphi _{\lambda _{0}}:\mathcal{A}%
_{\lambda _{0}}\rightarrow B(\mathcal{H})$ such that $\varphi =\varphi
_{\lambda _{0}}\circ \pi _{\lambda _{0}}^{\mathcal{A}}$. Moreover, since $%
\left. \varphi \right\vert _{\mathcal{S}}=\left. \pi \right\vert _{\mathcal{S%
}}$, we have $\left. \varphi _{\lambda _{0}}\right\vert _{\mathcal{S}%
_{\lambda _{0}}}=\left. \pi _{\lambda _{0}}\right\vert _{\mathcal{S}%
_{\lambda _{0}}}$, and since $\pi _{\lambda _{0}}$ is a boundary
representation for $\mathcal{S}_{\lambda _{0}}$, $\varphi _{\lambda
_{0}}=\pi _{\lambda _{0}}$. Therefore, $\varphi =\pi $, and consequently, $%
\pi $ is an admissible local boundary representation for $\mathcal{S}$.
\end{proof}

The collection of all admissible local boundary representations for a local
operator system $\mathcal{S}$ is called the \textit{admissible} \textit{%
local non-commutative Choquet boundary for $\mathcal{S}$, }and it is denoted
by\textit{\ }Ch$(\mathcal{S})$. We denote by Ch$_{\lambda }(\mathcal{S})$
the collection of all admissible local boundary representations $\pi :%
\mathcal{A}\rightarrow B(\mathcal{H})$ for $\mathcal{S}$ with the property
that $\left\Vert \pi \left( a\right) \right\Vert \leq p_{\lambda }\left(
a\right) $ for all $a\in \mathcal{A}$. If Ch$\left( \mathcal{S}_{\lambda
}\right) $ is the collection of all boundary representations for $\mathcal{S}%
_{\lambda }$, then, according to Proposition \ref{c}, there exists a
bijective correspondence between Ch$\left( \mathcal{S}_{\lambda }\right) $
and Ch$_{\lambda }(\mathcal{S})$.

\begin{proposition}
If $\mathcal{S}$ is separable, then it has sufficiently many admissible
local boundary representations, in the sense that, for each $\lambda \in
\Lambda $, for all $n\geq 1$ and for all $\left[ s_{ij}\right]
_{i,j=1}^{n}\in M_{n}\left( S\right) $, 
\begin{equation*}
p_{\lambda }^{n}\left( \left[ s_{ij}\right] _{i,j=1}^{n}\right) =\max
\left\{ \left\Vert \left[ \pi \left( s_{ij}\right) \right]
_{i,j=1}^{n}\right\Vert _{B(\mathcal{H}^{\oplus n})};\pi \in \text{Ch}%
_{\lambda }(\mathcal{S})\text{ }\right\} .
\end{equation*}
\end{proposition}

\begin{proof}
Let $\lambda \in \Lambda $. Since $\mathcal{S}$ is separable, $\mathcal{S}%
_{\lambda }$ is separable, and by \cite[Theorem 3.1]{K},%
\begin{eqnarray*}
p_{\lambda }^{n}\left( \left[ s_{ij}\right] _{i,j=1}^{n}\right)
&=&\left\Vert \left[ \pi _{\lambda }^{\mathcal{A}}\left( s_{ij}\right) %
\right] _{i,j=1}^{n}\right\Vert _{M_{n}(\mathcal{A}_{\lambda })} \\
&=&\max \left\{ \left\Vert \left[ \pi _{\lambda }\left( \pi _{\lambda }^{%
\mathcal{A}}\left( s_{ij}\right) \right) \right] _{i,j=1}^{n}\right\Vert _{B(%
\mathcal{H}^{\oplus n})};\pi _{\lambda }\in \text{Ch}\left( \mathcal{S}%
_{\lambda }\right) \right\} \\
&=&\max \left\{ \left\Vert \left[ \pi \left( s_{ij}\right) \right]
_{i,j=1}^{n}\right\Vert _{B(\mathcal{H}^{\oplus n})};\pi \in \text{Ch}%
_{\lambda }(\mathcal{S})\text{ }\right\}
\end{eqnarray*}%
for all for all $\left[ s_{ij}\right] _{i,j=1}^{n}\in M_{n}\left( S\right) $
and $n\geq 1$.
\end{proof}

\begin{remark}
A \textit{local boundary representation} for $\mathcal{S}$ is an irreducible
local representation $\pi :$ $\mathcal{A}\rightarrow $ $B(\mathcal{H})$ with
the property that it is the unique local completely positive extension of
the unital local $\mathcal{CP}$-map $\left. \pi \right\vert _{\mathcal{S}}$
to $\mathcal{A}$ \cite[Definition 3.9]{J24}. If $\pi :\mathcal{A}\rightarrow
B(\mathcal{H})$ is a local boundary representation for $\mathcal{S}$, then $%
\pi $ is an admissible local boundary representation for $\mathcal{S}$ \cite[%
Proposition 3.11]{J24}.
\end{remark}

\begin{example}
Let $\{X_{n};i_{nm}:X_{n}\hookrightarrow X_{m};n\leq m;n,m\in \mathbb{N}\}$
be an inductive system of Hausdorff compact spaces and $X:=\lim\limits_{%
\underset{n}{\rightarrow }}X_{n}$. Then $C(X):=\{f:X\rightarrow \mathbb{C};f$
is continuous$\}$ is a unital commutative \textit{Fr\'{e}chet locally }$%
C^{\ast }$-algebra with respect to the topology defined by the family of $%
C^{\ast }$-seminorms $\{p_{n}\}_{n\in \mathbb{N}}$, where $p_{n}\left(
f\right) =\sup \{\left\vert f\left( x\right) \right\vert ;x\in X_{n}\}$.
Moreover, $C(X)$ can be identified with $\lim\limits_{\underset{n}{%
\leftarrow }}C\left( X_{n}\right) $ \cite{Ph}. According to Proposition \ref%
{c}, we obtained that the admissible local Choquet boundary for $C(X)$ is $X$%
, and so, it coincides with the local Choquet boundary defined in \cite{J24}.
\end{example}

As in the case of operator systems, we show that the admissible local
boundary representations are intrinsic invariants for local operator systems
(see \cite[Theorem 4.2 ]{J24} and \cite[Theorem 2.1.2]{A}).

\begin{proposition}
\label{B} Let $\mathcal{A}_{1}$ and $\mathcal{A}_{2}$ be two unital locally $%
C^{\ast }$-algebras with the topologies given by the families of $C^{\ast }$%
-seminorms $\{p_{\lambda }\}_{\lambda \in \Lambda },$ respectively $%
\{q_{\lambda }\}_{\lambda \in \Lambda },$ $\mathcal{S}_{1}\subseteq \mathcal{%
A}_{1}$ and $\mathcal{S}_{2}\subseteq \mathcal{A}_{2}$ be two local operator
systems such that $\mathcal{S}_{1}$ generates $\mathcal{A}_{1}$ and $%
\mathcal{S}_{2}\ $generates $\mathcal{A}_{2}$. If $\varphi :\mathcal{S}%
_{1}\rightarrow \mathcal{S}_{2}$ is a surjective\textbf{\ }unital local
completely isometric map, then for each admissible local boundary
representation $\pi _{1}:\mathcal{A}_{1}\rightarrow B(\mathcal{H})$ for $%
\mathcal{S}_{1}$, there exists an admissible local boundary representation $%
\pi _{2}:\mathcal{A}_{2}\rightarrow B(\mathcal{H})$ for $\mathcal{S}_{2}$
such that $\pi _{2}\circ \varphi =\left. \pi _{1}\right\vert _{\mathcal{S}%
_{1}}$.
\end{proposition}

\begin{proof}
Let $\pi _{1}:\mathcal{A}_{1}\rightarrow B(\mathcal{H})$ be an admissible
boundary representation for $\mathcal{S}_{1}$. Then, there exist $\lambda
_{0}\in \Lambda $ and a boundary representation $\left( \pi _{1}\right)
_{\lambda _{0}}:\left( \mathcal{A}_{1}\right) _{\lambda _{0}}\rightarrow B(%
\mathcal{H})$ such that $\pi _{1}=\left( \pi _{1}\right) _{\lambda
_{0}}\circ \pi _{\lambda _{0}}^{\mathcal{A}_{1}}$. Since $\varphi :\mathcal{S%
}_{1}\rightarrow \mathcal{S}_{2}$ is a surjective unital local completely
isometric map, there exists a surjective unital completely isometric map $%
\varphi _{\lambda _{0}}:$ $\left( \mathcal{S}_{1}\right) _{\lambda
_{0}}\rightarrow \left( \mathcal{S}_{2}\right) _{\lambda _{0}}\ $such that $%
\varphi _{\lambda _{0}}\circ \left. \pi _{\lambda _{0}}^{\mathcal{A}%
_{1}}\right\vert _{\mathcal{S}_{1}}=\pi _{\lambda _{0}}^{\mathcal{A}%
_{2}}\circ \varphi $,$\ $and by \cite[Theorem 2.1.2]{A}, there exists a
boundary representation $\left( \pi _{2}\right) _{\lambda _{0}}:\left( 
\mathcal{A}_{2}\right) _{\lambda _{0}}\rightarrow B(\mathcal{H})$ such that $%
\left( \pi _{2}\right) _{\lambda _{0}}\circ \varphi _{\lambda _{0}}=\left.
\left( \pi _{1}\right) _{\lambda _{0}}\right\vert _{\left( \mathcal{S}%
_{1}\right) _{\lambda _{0}}}$. Then $\pi _{2}:=\left( \pi _{2}\right)
_{\lambda _{0}}\circ \pi _{\lambda _{0}}^{\mathcal{A}_{2}}$ is an admissible
local boundary representation for $\mathcal{S}_{2}$ and 
\begin{equation*}
\pi _{2}\circ \varphi =\left( \pi _{2}\right) _{\lambda _{0}}\circ \pi
_{\lambda _{0}}^{\mathcal{A}_{2}}\circ \varphi =\left( \pi _{2}\right)
_{\lambda _{0}}\circ \varphi _{\lambda _{0}}\circ \left. \pi _{\lambda
_{0}}^{\mathcal{A}_{1}}\right\vert _{\mathcal{S}_{1}}=\left. \left( \pi
_{1}\right) _{\lambda _{0}}\right\vert _{\left( \mathcal{S}_{1}\right)
_{\lambda _{0}}}\circ \left. \pi _{\lambda _{0}}^{\mathcal{A}%
_{1}}\right\vert _{\mathcal{S}_{1}}=\left. \pi _{1}\right\vert _{\mathcal{S}%
_{1}}.
\end{equation*}
\end{proof}

\subsection{The local Shilov boundary ideal for a local operator system}

Let $\mathcal{A}$ be a locally $C^{\ast }$-algebra whose topology is defined
by the family of $C^{\ast }$-seminorms $\{p_{\lambda }\}_{\lambda \in
\Lambda }$. Suppose that $\mathcal{I}$ is a\ closed two sided $\ast $-ideal
of $\mathcal{A}$. Then the quotient $\ast $-algebra $\mathcal{A}/\mathcal{I}$
is a $C^{\ast }$-convex algebra with respect to the family of $C^{\ast }$%
-seminorms $\left\{ \widehat{p}_{\lambda }\right\} _{\lambda \in \Lambda }$,
where $\widehat{p}_{\lambda }\left( a+\mathcal{I}\right) =\inf \{p_{\lambda
}\left( a+b\right) ;b\in \mathcal{I}\}$. If $\mathcal{A}$ is a Fr\'{e}chet
locally $C^{\ast }$-algebra{}, then $\mathcal{A}/\mathcal{I}$ is a locally Fr%
\'{e}chet $C^{\ast }$-algebra \cite[Theorem 2.7]{I}. Moreover, for each $%
\lambda \in \Lambda $, the $C^{\ast }$-algebras $\overline{\left( \mathcal{A}%
/\mathcal{I}\right) _{\lambda }}$ and $\mathcal{A}_{\lambda }/\mathcal{I}%
_{\lambda }$ are isomorphic, where $\overline{\left( \mathcal{A}/\mathcal{I}%
\right) _{\lambda }}$ is the completion of the normed $\ast $-algebra $%
\left( \mathcal{A}/\mathcal{I}\right) /\ker \widehat{p}_{\lambda }$ and $%
\mathcal{I}_{\lambda }$ is the closure of $\pi _{\lambda }^{\mathcal{A}%
}\left( \mathcal{I}\right) $ in $\mathcal{A}_{\lambda }$. Therefore, the
locally $C^{\ast }$-algebra $\overline{\mathcal{A}/\mathcal{I}}$ can be
identified with $\lim\limits_{\underset{\lambda }{\leftarrow }}\mathcal{A}%
_{\lambda }/\mathcal{I}_{\lambda }$, and the canonical $\ast $-morphism $%
\sigma :\mathcal{A}\rightarrow \mathcal{A}/\mathcal{I}$ is an inverse limit $%
\ast $-morphism, $\sigma =\lim\limits_{\underset{\lambda }{\leftarrow }%
}\sigma _{\lambda }$, where $\sigma _{\lambda }$ is the canonical morphism
from $\mathcal{A}_{\lambda }$ onto $\mathcal{A}_{\lambda }/\mathcal{I}%
_{\lambda }\ $for all $\lambda \in \Lambda $.

Let $\mathcal{S}\subseteq \mathcal{A}$ be a local operator system such that
the locally $C^{\ast }$-subalgebra $C^{\ast }(\mathcal{S})$ of $\mathcal{A}$
generated by $\mathcal{S}$ coincides with $\mathcal{A}$.

\begin{definition}
A closed two sided $\ast $-ideal $\mathcal{I}$ of $\mathcal{A}$ is called a
local boundary ideal for $\mathcal{S}$ if the canonical map $\sigma :%
\mathcal{A}\rightarrow \mathcal{A}/\mathcal{I}$ is local completely
isometric on $\mathcal{S}$. A local boundary ideal for $\mathcal{S}$ is
called the local Shilov boundary ideal for $\mathcal{S}$ if it contains
every other local boundary ideal.
\end{definition}

We will show that for every separable local operator system the local Shilov
boundary ideal exists. Moreover, it is unique.

\begin{theorem}
\label{A}Let $\mathcal{A}$ be a unital separable locally $C^{\ast }$-algebra
whose topology is given the family of $C^{\ast }$-seminorms $\{p_{\lambda
}\}_{\lambda \in \Lambda }$ and $\mathcal{S}\subseteq \mathcal{A}$ be a
local operator system such that the locally $C^{\ast }$-subalgebra $C^{\ast
}(\mathcal{S})$ of $\mathcal{A}$ generated by $\mathcal{S}$ coincides with $%
\mathcal{A}$. Then $\mathcal{J}=\bigcap\limits_{\pi \in \text{Ch}(\mathcal{S}%
)}\ker \pi $ is the local Shilov boundary ideal for $\mathcal{S}$.
\end{theorem}

\begin{proof}
First, we show that $\mathcal{J}$ is a local boundary ideal for $\mathcal{S}$%
. Let $\lambda \in \Lambda $. Since $\mathcal{S}$ is separable, $\mathcal{S}%
_{\lambda }$ is separable, and by \cite[Theorem 2.2.3]{A}, $\mathcal{J}%
_{\lambda }=\bigcap\limits_{\pi _{\lambda }\in \text{Ch}(\mathcal{S}%
_{\lambda })}\ker \pi _{\lambda }$ is the Shilov boundary ideal for $%
\mathcal{S}_{\lambda }$, and 
\begin{equation*}
\left\Vert \left[ \pi _{\lambda }^{\mathcal{A}}\left( s_{ij}\right) \right]
_{i,j=1}^{n}\right\Vert _{M_{n}(\mathcal{A}_{\lambda })}=\left\Vert \left[
\pi _{\lambda }^{\mathcal{A}}\left( s_{ij}\right) +\mathcal{J}_{\lambda }%
\right] _{i,j=1}^{n}\right\Vert _{M_{n}(\mathcal{A}_{\lambda }/\mathcal{J}%
_{\lambda })}
\end{equation*}%
for all $\left[ s_{ij}\right] _{i,j=1}^{n}\in M_{n}\left( \mathcal{S}\right) 
$ and for all $n\geq 1$. If $a\in \mathcal{J}$, then $\pi _{\lambda }\left(
\pi _{\lambda }^{\mathcal{A}}\left( a\right) \right) =\left( \pi _{\lambda
}\circ \pi _{\lambda }^{\mathcal{A}}\right) \left( a\right) =0$ for all
boundary representation $\pi _{\lambda }$ for $\mathcal{S}_{\lambda }$ (see
Proposition \ref{c}). Therefore, $\overline{\pi _{\lambda }^{\mathcal{A}%
}\left( \mathcal{J}\right) }\subseteq \mathcal{J}_{\lambda }$, and then 
\begin{eqnarray*}
p_{\lambda }^{n}\left( \left[ s_{ij}\right] _{i,j=1}^{n}\right)
&=&\left\Vert \left[ \pi _{\lambda }^{\mathcal{A}}\left( s_{ij}\right) %
\right] _{i,j=1}^{n}\right\Vert _{M_{n}(\mathcal{A}_{\lambda })}=\left\Vert %
\left[ \pi _{\lambda }^{\mathcal{A}}\left( s_{ij}\right) +\mathcal{J}%
_{\lambda }\right] _{i,j=1}^{n}\right\Vert _{M_{n}(\mathcal{A}_{\lambda }/%
\mathcal{J}_{\lambda })} \\
&\leq &\left\Vert \left[ \pi _{\lambda }^{\mathcal{A}}\left( s_{ij}\right) +%
\overline{\pi _{\lambda }^{\mathcal{A}}\left( \mathcal{J}\right) }\right]
_{i,j=1}^{n}\right\Vert _{M_{n}(\mathcal{A}_{\lambda }/\overline{\pi
_{\lambda }^{\mathcal{A}}\left( \mathcal{J}\right) })}=\widehat{p}_{\lambda
}^{n}\left( \left[ s_{ij}+\mathcal{J}\right] _{i,j=1}^{n}\right) \\
&\leq &p_{\lambda }^{n}\left( \left[ s_{ij}\right] _{i,j=1}^{n}\right)
\end{eqnarray*}%
for all $\left[ s_{ij}\right] _{i,j=1}^{n}\in M_{n}\left( \mathcal{S}\right) 
$ and for all $n\geq 1$. Therefore,%
\begin{equation*}
\widehat{p}_{\lambda }^{n}\left( \left[ s_{ij}+\mathcal{J}\right]
_{i,j=1}^{n}\right) =p_{\lambda }^{n}\left( \left[ s_{ij}\right]
_{i,j=1}^{n}\right)
\end{equation*}%
for all $\left[ s_{ij}\right] _{i,j=1}^{n}\in M_{n}\left( \mathcal{S}\right) 
$ and for all $n\geq 1$. Consequently, $\mathcal{J}$ is a local boundary
ideal.

Now, show that $\mathcal{J}$ is the largest local boundary ideal for $%
\mathcal{S}$. First, we note that by Proposition \ref{c}, 
\begin{equation*}
\mathcal{J}=\bigcap\limits_{\pi \in \text{Ch}(\mathcal{S})}\ker \pi
=\bigcap\limits_{\lambda \in \Lambda }\bigcap\limits_{\pi \in \text{Ch}%
_{\lambda }(\mathcal{S})}\ker \pi .
\end{equation*}%
If $\mathcal{I}$ is a local boundary ideal for $\mathcal{S}$, then, since
for each $\lambda \in \Lambda $, the canonical map from $\mathcal{A}%
_{\lambda }$ to $\mathcal{A}_{\lambda }/$ $\overline{\pi _{\lambda }^{%
\mathcal{A}}\left( \mathcal{I}\right) }$ is completely isometric, for each $%
\lambda \in \Lambda ,$ $\overline{\pi _{\lambda }^{\mathcal{A}}\left( 
\mathcal{I}\right) }$ is a boundary ideal for $\mathcal{S}_{\lambda },$ and
so, $\overline{\pi _{\lambda }^{\mathcal{A}}\left( \mathcal{I}\right) }%
\subseteq \mathcal{J}_{\lambda }$. Consequently, $\mathcal{I}\subseteq \ker
\left( \pi _{\lambda }\circ \pi _{\lambda }^{\mathcal{A}}\right) $ for all
boundary representation $\pi _{\lambda }$ of $\mathcal{S}_{\lambda }$ and
for all $\lambda \in \Lambda $. Therefore, $\mathcal{I\subseteq J}$, and so $%
\mathcal{J}$ is the largest local boundary ideal for $\mathcal{S}$.
\end{proof}

\begin{proposition}
\label{C} Let $\mathcal{A}$ be a unital separable locally $C^{\ast }$%
-algebra whose topology is given the family of $C^{\ast }$-seminorms $%
\{p_{\lambda }\}_{\lambda \in \Lambda },\mathcal{S}\subseteq \mathcal{A}$ be
a local operator system such that it generates $\mathcal{A}$ and $\mathcal{J}
$ be the local Shilov boundary for $\mathcal{S}$. If $\sigma $ is the
canonical map from $\mathcal{A}$ to $\mathcal{A}/\mathcal{J}$, then $\sigma
\left( \mathcal{S}\right) $ is a separable local operator system in $%
\overline{\mathcal{A}/\mathcal{J}}$ whose local Shilov boundary ideal is $%
\{0\}$.
\end{proposition}

\begin{proof}
Clearly, $\sigma \left( \mathcal{S}\right) $ is a separable local operator
system in $\overline{\mathcal{A}/\mathcal{J}}$ and generates $\overline{%
\mathcal{A}/\mathcal{J}}$ . Using Proposition \ref{B} and Theorem \ref{A},
in a similar way to the proof of \cite[Proposition 2.2.4]{A}, we obtain that
the local Shilov boundary ideal for $\ \sigma \left( \mathcal{S}\right) $ is
trivial.
\end{proof}

A separable operator system determines the structure of its generated $%
C^{\ast }$-algebra, once one has factored by the Shilov boundary ideal\ \cite%
[Theorem 2.2.5]{A}. In the case of separable Fr\'{e}chet local operator
systems, we obtain similar result.

\begin{theorem}
\label{D}Let $\mathcal{A}_{1}$ and $\mathcal{A}_{2}$ be two separable unital
Fr\'{e}chet locally $C^{\ast }$-algebras whose topologies are given by the
families of $C^{\ast }$-seminorms $\{p_{n}\}_{n\geq 1}$, respectively $%
\{q_{n}\}_{n\geq 1},$ $\mathcal{S}_{1}\subseteq \mathcal{A}_{1}$ and $%
\mathcal{S}_{2}\subseteq \mathcal{A}_{2}$ be two local operator systems such
that $\mathcal{S}_{1}$ generates $\mathcal{A}_{1}$ and $\mathcal{S}_{2}$
generates $\mathcal{A}_{2}$. Suppose that both $\mathcal{S}_{1}\ $and $%
\mathcal{S}_{2}$ have trivial local Shilov boundary ideals. Then any
surjective unital local completely isometric map from $\mathcal{S}_{1}\ $to $%
\mathcal{S}_{2}$ is implemented by a unital local contractive $\ast $%
-isomorphism from $\mathcal{A}_{1}\ $to $\mathcal{A}_{2}.$
\end{theorem}

\begin{proof}
Suppose that $\varphi :\mathcal{S}_{1}\rightarrow $ $\mathcal{S}_{2}$ is a
unital surjective local completely isometric map. We denote by $\mathcal{BR}%
_{n}\left( \mathcal{S}_{i}\right) $, the set of all equivalence classes of
admissible boundary representations $\left( \pi _{i}\right) _{n}$ for $%
\mathcal{S}_{i},$ $i=1,2,$ such that $\left\Vert \left( \pi _{1}\right)
_{n}\left( a_{1}\right) \right\Vert \leq p_{n}\left( a_{1}\right) $ for all $%
a_{1}\in \mathcal{A}_{1}$, respectively $\left\Vert \left( \pi _{2}\right)
_{n}\left( a_{2}\right) \right\Vert \leq q_{n}\left( a_{2}\right) $ for all $%
a_{2}\in \mathcal{A}_{2}$. By Proposition \ref{B}, for each $\left( \pi
_{1}\right) _{n}\in \mathcal{BR}_{n}\left( \mathcal{S}_{1}\right) ,$ there
exists $\left( \pi _{2}\right) _{n}\in \mathcal{BR}_{n}\left( \mathcal{S}%
_{2}\right) $ such that $\left( \pi _{2}\right) _{n}\circ \varphi =\left.
\left( \pi _{1}\right) _{n}\right\vert _{\mathcal{S}_{1}}$. Let 
\begin{equation*}
\left( \Psi _{1}\right) _{n}:=\bigoplus\limits_{\left( \pi _{1}\right)
_{n}\in \mathcal{BR}_{n}\left( \mathcal{S}_{1}\right) }\left( \pi
_{1}\right) _{n}\text{ and }\left( \Psi _{2}\right)
_{n}:=\bigoplus\limits_{\left( \pi _{2}\right) _{n}\in \mathcal{BR}%
_{n}\left( \mathcal{S}_{2}\right) }\left( \pi _{2}\right) _{n}.
\end{equation*}
Clearly, $\left( \Psi _{1}\right) _{n}$ and $\left( \Psi _{2}\right) _{n}$
are local representations of $\mathcal{A}_{1}$, respectively $\mathcal{A}%
_{2} $ on a Hilbert space $H_{n}$ and $\left( \Psi _{2}\right) _{n}\circ
\varphi =\left. \left( \Psi _{1}\right) _{n}\right\vert _{\mathcal{S}_{1}}.$

For each $n\geq 1$, set $\mathcal{H}_{n}:=\bigoplus\limits_{m\leq n}H_{m}$,
and for each $a_{i}\in \mathcal{A}_{i}$, consider the map $\left( \rho
_{i}\right) _{n}\left( a_{i}\right) :\mathcal{H}_{n}\rightarrow \mathcal{H}%
_{n}$ defined by 
\begin{equation*}
\left( \rho _{i}\right) _{n}\left( a_{i}\right) \left(
\bigoplus\limits_{m\leq n}\xi _{m}\right) =\bigoplus\limits_{n\leq m}\left(
\Psi _{i}\right) _{m}\left( a_{i}\right) \left( \xi _{m}\right)
\end{equation*}%
$i=1,2$. Clearly, $\left( \rho _{i}\right) _{n}$ is a local representation
of $\mathcal{A}_{i}$ on $\mathcal{H}_{n},i=1,2$, and $\left( \rho
_{2}\right) _{n}\circ \varphi =\left. \left( \rho _{1}\right)
_{n}\right\vert _{\mathcal{S}_{1}}.$

Moreover, the family $\{\mathcal{H}_{n}\}_{n\geq 1}$ of Hilbert spaces is
directed and its union $\mathcal{D}_{\mathcal{E}}$ is a pre-Hilbert space.
If $\mathcal{H} $ is the Hilbert space obtained by the completion of the
pre-Hilbert space $\mathcal{D}_{\mathcal{E}}$, then $\{\mathcal{H},\mathcal{E%
}=\{\mathcal{H}_{n}\}_{n\geq 1},\mathcal{D}_{\mathcal{E}}\}$ is a Fr\'{e}%
chet quantized domain, and for each $a_{i}\in \mathcal{A}_{i},$ the map $%
\rho _{i}\left( a_{i}\right) :\mathcal{D}_{\mathcal{E}}\rightarrow \mathcal{D%
}_{\mathcal{E}}$ defined by 
\begin{equation*}
\rho _{i}\left( a_{i}\right) \eta _{n}=\left( \rho _{i}\right) _{n}\left(
a_{i}\right) \eta _{n},\eta _{n}\in \mathcal{H}_{n}
\end{equation*}%
is an element in $C^{\ast }\left( \mathcal{D}_{\mathcal{E}}\right) ,$ $i=1,2$%
. In this way, we obtain an admissible\textbf{\ }local contractive $\ast $%
-morphiam $\rho _{i}:\mathcal{A}_{i}\rightarrow C^{\ast }\left( \mathcal{D}_{%
\mathcal{E}}\right) ,i=1,2.$ Moreover, $\rho _{2}\circ \varphi =\left. \rho
_{1}\right\vert _{\mathcal{S}_{1}}$. Since both $\mathcal{S}_{1}\ $and $%
\mathcal{S}_{2}$ have trivial local Shilov boundary ideals, 
\begin{equation*}
\bigcap\limits_{n\geq 1}\bigcap\limits_{\left( \psi _{1}\right) _{n}\in 
\mathcal{BR}_{n}\left( \mathcal{S}_{1}\right) }\ker \left( \psi _{1}\right)
_{n}=\{0\}=\bigcap\limits_{n\geq 1}\bigcap\limits_{\left( \psi _{2}\right)
_{n}\in \mathcal{BR}_{n}\left( \mathcal{S}_{2}\right) }\ker \left( \psi
_{2}\right) _{n}.
\end{equation*}%
Therefore, $\rho _{1}$ and $\rho _{2}$ are injective admissible local
contractive $\ast $-morphisms. Moreover, 
\begin{equation*}
\overline{\rho _{2}\left( \mathcal{A}_{2}\right) }=\overline{\rho _{2}\left(
C^{\ast }(\mathcal{S}_{2}\right) }=\overline{\rho _{2}\left( C^{\ast
}(\varphi \left( \mathcal{S}_{1}\right) \right) }=C^{\ast }(\rho _{1}\left( 
\mathcal{S}_{1}\right) )=\overline{\rho _{1}\left( \mathcal{A}_{1}\right) }
\end{equation*}%
and then, $\rho _{2}:\mathcal{A}_{2}\rightarrow \overline{\rho _{1}\left( 
\mathcal{A}_{1}\right) }$ is an invertible unital local contractive $\ast $%
-morphism. Since $\overline{\rho _{1}\left( \mathcal{A}_{1}\right) }$ is a Fr%
\'{e}chet locally $C^{\ast }$-algebra, $\rho _{2}^{-1}:\overline{\rho
_{1}\left( \mathcal{A}_{1}\right) }\ \rightarrow \mathcal{A}_{2}$ is a
unital local contractive $\ast $-morphism. Let $\widetilde{\varphi }:=\rho
_{2}^{-1}\circ \rho _{1}$. Clearly, $\widetilde{\varphi }$ is a unital%
\textbf{\ }local contractive $\ast $-isomorphism from $\mathcal{A}_{1}$ to $%
\mathcal{A}_{2}$, and 
\begin{equation*}
\left. \widetilde{\varphi }\right\vert _{\mathcal{S}_{1}}=\rho
_{2}^{-1}\circ \left. \rho _{1}\right\vert _{\mathcal{S}_{1}}=\varphi \text{.%
}
\end{equation*}
\end{proof}

\end{document}